\documentclass[a4paper,11pt]{amsart}

\title[Going to Lorentz when fractional estimates fail]{Going to Lorentz when fractional Sobolev, Gagliardo and Nirenberg estimates fail}

\author{Ha\"\i m Brezis}
\address{Department of Mathematics\\
  Rutgers University, Hill Center, Busch Campus\\ 
  110 Frelinghuysen Road, Piscataway, NJ 08854, USA}
\address{
  Departments of Mathematics and Computer Science\\ Technion, Israel Institute of Technology\\ 32.000 Haifa, Israel}
\address{
  Laboratoire Jacques-Louis Lions\\
  Sorbonne Universit\'es, UPMC Universit\'e Paris-6, 4  place Jussieu\\
  75005 Paris, France}
\email{brezis@math.rutgers.edu}

\author{Jean Van Schaftingen}
\address{Universit\'e catholique de Louvain\\ 
Institut de Recherche en Math\'ematique et Physique\\
Chemin du Cyclotron 2 bte L7.01.01\\
1348 Louvain-la-Neuve\\
Belgium}
\email{Jean.VanSchaftingen@UCLouvain.be}

\author{Po-Lam Yung}
\address{Mathematical Sciences Institute \\
Australian National University \\
Canberra ACT 2601 \\
Australia} 
\email{PoLam.Yung@anu.edu.au}
\address{Department of Mathematics \\
The Chinese University of Hong Kong \\
Ma Liu Shui \\
Hong Kong} 
\email{plyung@math.cuhk.edu.hk}

\usepackage[T1]{fontenc}
\usepackage[utf8]{inputenc}
\usepackage{geometry}

\usepackage{lmodern} 
\usepackage{microtype}

\usepackage{amssymb}
\usepackage{amsmath}
\usepackage{amsthm}
\usepackage{esint}
\usepackage{comment}

\usepackage{constants}
\usepackage{paralist}
\usepackage{mathrsfs}
\usepackage{mathtools}
\newcommand{\defeq}{\coloneqq}

\usepackage{todonotes}

\usepackage[pdftitle={Going to Lorentz when Gagliardo--Nirenberg fails},%
hidelinks,
pdfauthor={Haïm Brezis, Jean Van Schaftingen \& Po-Lam Yung},%
final]{hyperref}

\usepackage[nameinlink%
,capitalize%
]{cleveref}
 
\usepackage[abbrev,backrefs]{amsrefs}

\newtheorem{theorem}{Theorem}
\newtheorem{proposition}{Proposition}[section]
\newtheorem{lemma}[proposition]{Lemma}

\theoremstyle{definition}

\newtheorem{openproblem}{Open Problem}

\theoremstyle{remark}
\newtheorem{remark}[proposition]{Remark}

\numberwithin{equation}{section}

\newcommand{\abs}[1]{{\lvert #1 \rvert}}

\newcommand{\norm}[2][]{{\lVert #2 \rVert}_{#1}}
\newcommand{\seminorm}[2][]{{\lvert #2 \rvert}_{#1}}

\newcommand{\Biggnorm}[2][]{{\Biggl\lVert #2 \Biggr\rVert}_{#1}}

\newcommand{\quasinorm}[2][]{{[ #2 ]}_{#1}}

\newcommand{\Biggquasinorm}[2][]{{\Biggl[ #2 \Biggr]}_{#1}}

\newcommand{\st}{\;:\;}

\newcommand{\Rset}{\mathbb{R}}

\newcommand{\dif}{\,\mathrm{d}}

\newcommand{\eofs}{\,}

 \makeatletter
\renewcommand\subsection{\@startsection{subsection}{2}%
   \z@{.5\linespacing\@plus.7\linespacing}{.1\linespacing}%
   {\normalfont\itshape}} 
\makeatother

\begin{document}

\begin{abstract}
In the cases where there is no Sobolev-type or Gagliardo--Nirenberg-type fractional estimate involving $\abs{u}_{W^{s,p}}$, we establish alternative estimates where the strong $L^p$ norms are replaced by Lorentz norms. 
\end{abstract}
\maketitle

\section{Introduction}

In \cite{Brezis_VanSchaftingen_Yung_2021}*{Theorem 1.1}, it was shown that there exists a constant \(C = C(N)\) such that 
\begin{equation}
  \label{eq_chaK6eique5lo9ahy9Ooju9e}
    \Biggquasinorm[M^p (\Rset^N \times \Rset^N)]{\frac{u (x) - u (y)}{\abs{x - y}^{\frac{N}{p} + 1}}}
    \le 
    C^{1/p} 
    \,
    \norm[L^p (\Rset^N)]{\nabla u}\eofs,\quad \forall u \in C^\infty_c (\Rset^N), \forall p \ge 1.
\end{equation}
Here \(M^p (\Rset^N \times \Rset^N) = L^p_w (\Rset^N \times \Rset^N) = L^{p, \infty} (\Rset^N \times \Rset^N)\), \(1 \le p <\infty\), is the Marcinkiewicz (=weak \(L^p\)) space modelled on \(L^p (\Rset^N \times \Rset^N)\), defined by the condition 
\[
\quasinorm[M^p(\Rset^N \times \Rset^N)]{f}^p \defeq \sup_{\lambda > 0} \lambda^p \mathcal{L}^{2N} \bigl(\{ (x,y)\in \Rset^N \times \Rset^N \st \abs{f (x)} \ge \lambda\}\bigr)
< \infty
\]
(see for example \cite{Castillo_Rafeiro_2016}*{Chapter 5} or \cite{Grafakos_2014}*{\S 1.1}).

We also know that the inequality
\begin{equation}
\label{eq_reiG4eiboodieheu4Ohnoh1U} \Biggnorm[L^p (\Rset^N \times \Rset^N)]{\frac{u (x) - u (y)}{\abs{x - y}^{\frac{N}{p} + 1}}}\le 
    C(N,p) 
    \,
    \norm[L^p (\Rset^N)]{\nabla u}\eofs,\quad \forall u \in C^\infty_c (\Rset^N), \forall p \ge 1
\end{equation}
does \emph{not} hold. 
In fact the failure of \eqref{eq_reiG4eiboodieheu4Ohnoh1U} is more striking: for every \(1 \le p < \infty\) and every measurable function \(u\),
\begin{equation}
\label{eq_oofohxeefou8ohShev1zaafe}
\Biggnorm[L^p (\Rset^N \times \Rset^N)]{\frac{u (x) - u (y)}{\abs{x - y}^{\frac{N}{p} + 1}}}
< \infty \quad \Longrightarrow \quad u \text{ is constant};
\end{equation}
See \cite{Bourgain_Brezis_Mironescu_2000} and \cites{Brezis_2002,DeMarco_Mariconda_Solimini_2008,RanjbarMotlagh}.

A natural question is whether one can improve \eqref{eq_chaK6eique5lo9ahy9Ooju9e} in the Lorentz scale, where \(L^{p, q} (X, \mu)\), with \(1 \le p <\infty\) and \(1 \le q\le \infty\), is characterized by (see for example \cite{Grafakos_2014}*{\S 1.4}, \cite{Castillo_Rafeiro_2016}*{Chapter 6}, \cite{Hunt_1966} or \cite{Ziemer_1989}*{\S 1.8}), when \(q < \infty\)
\begin{equation}
  \label{eq_chahvoChohyae4aiK9ez4zoh}
\quasinorm[L^{p, q} (\Rset^N \times \Rset^N)]{f}^q
= p \int_0^\infty \lambda^q \, \mathcal{L}^{2N} (\{ (x, y) \in \Rset^N \times \Rset^N \st \abs{f (x, y)} \ge \lambda\})^\frac{q}{p} \frac{\dif \lambda}{\lambda} < +\infty\eofs,
\end{equation}
and when \(q = \infty\) by \(\quasinorm[L^{p, \infty} (\Rset^N\times \Rset^N)]{f}
= \quasinorm[M^p (\Rset^N \times \Rset^N)]{f} <+\infty\).
In other words, the question is whether the estimate
\begin{equation}
    \Biggquasinorm[L^{p, q} (\Rset^N \times \Rset^N)]{\frac{u (x) - u (y)}{\abs{x - y}^{\frac{N}{p} + 1}}}
    \le 
    C(N, p, q)
    \,
    \norm[L^p (\Rset^N)]{\nabla u}\eofs,\quad \forall u \in C^\infty_c (\Rset^N)
  \end{equation}
  holds for some \(q \in (p,\infty)\) (\(q\) depending on \(p\) and \(N\)).
  The answer is negative, as can be seen from the following generalization of \eqref{eq_oofohxeefou8ohShev1zaafe}.
  
  \begin{theorem}
  \label{thm_X_1} Assume that \(1 \le p < \infty\),
  \(1 \le q < \infty\) and \(u\) is measurable. Then
  \begin{equation}
  \label{eq_sebaich1Ofai9Ooqu2bainui}
  \Biggquasinorm[L^{p, q} (\Rset^N \times \Rset^N)]{\frac{u (x) - u(y)}{\abs{x - y}^{\frac{N}{p} + 1}} }
  < \infty \quad \Longrightarrow \quad u \text{ is constant}.
  \end{equation}
  \end{theorem}

The heart of the matter is the following far-reaching extension of \eqref{eq_oofohxeefou8ohShev1zaafe}. It was originally presented in \cite[Proposition 6.3]{Brezis_VanSchaftingen_Yung_2020_arXiv} when \(p > 1\); the case \(p = 1\) is essentially due to A. Poliakovsky \cite[Corollary 1.1]{Poliakovsky} who settled \cite[Open Problem 1]{Brezis_VanSchaftingen_Yung_2020_arXiv}.  
  
\begin{theorem}
\label{proposition_level_set_constant}
Let \(1 \le  p < \infty\) and let  \(u : \Rset^N \to \Rset\) be a measurable function satisfying
\begin{equation}
\label{eq_EuK1che4iepahtheyoo2oiw3}
 \lim_{\lambda \to \infty} 
 \lambda^p 
 \mathcal{L}^{2N}
 \biggl(
 \biggl\{
 (x, y) \in \Rset^N \times \Rset^N
 \st 
 \frac{\abs{u (x) - u (y)}}{\abs{x - y}^{\frac{N}{p} + 1}}
 \ge \lambda
 \biggr\}
 \biggr)
  = 0\eofs .
\end{equation}
Then \(u\) is constant. 
\end{theorem}

The proofs of \cref{thm_X_1,proposition_level_set_constant} are given in \cref{section_constant}.

\medskip

Recall the fractional Sobolev spaces \(W^{s, p}\) (also called Slobodeskii spaces) is associated with the Gagliardo semi-norm, \(0 < s < 1\) and \(1 \le p < \infty\) defined by 
\begin{equation}
\label{eq_Jai7ahg2tahz4ua3zaishatu}
\abs{u}_{W^{s, p}}^p 
\defeq \iint\limits_{\Rset^N \times \Rset^N} \frac{\abs{u (x) - u (y)}^p}{\abs{x - y}^{N + sp}} \dif y \dif x\eofs.
\end{equation}
In \cite{Brezis_VanSchaftingen_Yung_2021} we announced the following theorem, which is a substitute for a fractional Sobolev-type estimate \(\dot{W}^{1,1}(\Rset^N) \hookrightarrow W^{1-N(1-\frac{1}{p}),p}(\Rset^N)\) that fails in dimension \(N = 1\):

\begin{theorem}
\label{theorem_W11_Wsp}\cite{Brezis_VanSchaftingen_Yung_2021}*{Corollary 4.1}
There exists an absolute constant \(C\) such that for every \(1 \leq p < \infty\), 
\begin{equation}
  \label{eq_Queewohz0oozoghie9rohjai}
    \Biggquasinorm[M^p (\Rset \times \Rset)]{\frac{u (x) - u (y)}{\abs{x - y}^{\frac{2}{p}}}}
    \le 
    C 
    \,
    \norm[L^1 (\Rset)]{u'}\eofs,\qquad \forall u \in C^\infty_c (\Rset)\eofs.
\end{equation}
\end{theorem}

\begin{remark}
When \(p = 2\), estimate \eqref{eq_Queewohz0oozoghie9rohjai} is originally due to Greco and Schiattarella \cite{Greco_Schiattarella}. 
\end{remark}

In \cite{Brezis_VanSchaftingen_Yung_2021}, we also announced the following theorem, which offers an alternative when the ``anticipated'' fractional Gagliardo--Nirenberg-type inequality 
\begin{equation}
  \Biggnorm[L^p (\Rset^N \times \Rset^N)]
  {\frac{u (x) - u (y)}{\abs{x - y}^\frac{N + 1}{p}}}
  \le C \norm[L^\infty  (\Rset^N)] {u}^{1 - 1/p}
  \norm[L^{1} (\Rset^N)]{\nabla u}^{1/p}\eofs, \qquad 
  \forall u \in C^\infty_c (\Rset^N)
\end{equation}
fails for every \(1 \le p < \infty\):

\begin{theorem}
  \label{theorem_intro}\cite{Brezis_VanSchaftingen_Yung_2021}*{Corollary 5.1}
For every \(N \geq 1\), there exists a constant \(C = C(N)\) such that for every \(1 \le p < \infty\),
\begin{equation} 
\label{eq:1.14}
\Biggquasinorm[M^p (\Rset^N \times \Rset^N)]
{\frac{u (x) - u (y)}{\abs{x - y}^\frac{N + 1}{p}}}\le C \norm[L^\infty  (\Rset^N)]{u}^{1 - 1/p}
\norm[L^1 (\Rset^N)]{\nabla u}^{1/p}\eofs,
\qquad \forall  u \in C^\infty_c (\Rset^N)\eofs.
\end{equation}
\end{theorem}

\Cref{theorem_intro} clearly implies \cref{theorem_W11_Wsp} since \(\norm[L^\infty(\Rset)]{u} \le \norm[L^1 (\Rset)]{u'},\,\forall u \in C^\infty_c (\Rset)\). 
The case \(p = 1\) of \eqref{eq:1.14} follows from \eqref{eq_chaK6eique5lo9ahy9Ooju9e} by setting \(p = 1\). The proof of \eqref{eq:1.14} in the case \(p > 1\) can be reduced to the case \(p = 1\) (see \cref{section_weak_estimates}). Alternatively, in the same section, we show how we can establish \eqref{eq:1.14} in the case \(p > 1\) by a more elementary method, if we allow a constant \(C\) that depends not only on \(N\) but also on \(p\). This method, which is in line with the techniques in Bourgain, Brezis and Mironescu \cite{Bourgain_Brezis_Mironescu_2001}, can be contrasted with the proof of \eqref{eq:1.14} in the case \(p = 1\): the latter relies on a covering argument in the one-dimensional case and the method of rotation to reach higher dimensions.

\Cref{theorem_W11_Wsp,theorem_intro} can be restated equivalently as Lorentz spaces estimates. 
In particular, replacing \(M^p\) by \(L^{p, \infty}\), we get \cref{theorem_W11_Wsp,theorem_intro} at the endpoint of the Lorentz scale. 

We show in \cref{section_optimal} that there is no improvement of \cref{theorem_W11_Wsp,theorem_intro} in the Lorentz scale. (Recall that for any fixed \(p\) the Lorentz spaces \(L^{p, q}\) increase as \(q\) increases.)

\medskip 

We now turn to another situation, also involving \(\dot{W}^{1, 1}\), where a Gagliardo--Nirenberg-type inequality fails. 
Let \(0 < s_1 < 1\), \(1 < p_1 < \infty\) and \(0 < \theta < 1\).
Set
\begin{align}
  \label{eq_Ahqueep9eci9aikio5eequai}
  s &= \theta s_1 + (1 - \theta) &
  &\text{ and }&
  \frac{1}{p} & = \frac{\theta}{p_1} + (1 - \theta)\eofs.
\end{align}
It is known that  the estimate 
\begin{equation}
  \label{eq_chuofah2shaiZiemohqu5eej}
  \seminorm[W^{s, p} (\Rset^N)]{u}
  =
  \Biggnorm[L^p (\Rset^N \times \Rset^N)]{\frac{u (x) - u (y)}{\abs{x - y}^{\frac{N}{p} + s}}}
  \le 
  C \seminorm[W^{s_1, p_1} (\Rset^N)]{u}^\theta \norm[L^1 (\Rset^N)]{\nabla u}^{1 - \theta}
\end{equation}
\begin{itemize}
  \item \emph{holds} for every \(\theta \in (0, 1)\) when \(s_1 p_1 < 1\) (Cohen, Dahmen, Daubechies and DeVore \cite{Cohen_Dahmen_Daubechies_DeVore_2003}),
  \item \emph{fails} for every \(\theta \in (0, 1)\) when \(s_1 p_1 \ge 1\) (Brezis and Mironescu \cite{Brezis_Mironescu_2018}).
\end{itemize}
We investigate here what happens in the regime \(s_1 p_1 \ge 1\).
Our main result in this direction is

\begin{theorem}

  \label{theorem_Lorentz} 
For every $N \geq 1$, \(p_1 \in (1, \infty)\) and \(\theta \in (0, 1)\), there exists a constant \(C = C(N,p_1,\theta)\) such that for all \(s_1 \in (0, 1)\) with \(s_1 p_1 \ge 1\), we have 
  \begin{equation}
    \label{eq_giu7ohWoo2vae4jei0AhgieG}
    \Biggquasinorm[L^{p, \frac{p_1}{\theta}} (\Rset^N \times \Rset^N)]{\frac{u (x) - u (y)}{\abs{x - y}^{\frac{N}{p} + s}}}
    \le 
    C \seminorm[W^{s_1, p_1} (\Rset^N)]{u}^\theta \norm[L^1 (\Rset^N)]{\nabla u}^{1 - \theta},
    \qquad \forall u \in C^\infty_c (\Rset^N) \eofs,
  \end{equation}
where \(0 < s <1\) and \(1 < p<\infty\) are defined by \eqref{eq_Ahqueep9eci9aikio5eequai}.
\end{theorem}

Note that \(L^{p} (\Rset^N \times \Rset^N) \subsetneq L^{p, \frac{p_1}{\theta}} (\Rset^N \times \Rset^N)\) since \(p < \frac{p_1}{\theta}\); this is consistent with the fact that \eqref{eq_chuofah2shaiZiemohqu5eej} fails when \(s_1 p_1 \ge 1\). As an immediate consequence of \eqref{eq_giu7ohWoo2vae4jei0AhgieG} we obtain 
  \begin{equation}
  \label{eq_phai1AichooTheeRah3usi2a}
    \Biggquasinorm[M^p (\Rset^N \times \Rset^N)]{\frac{u (x) - u (y)}{\abs{x - y}^{\frac{N}{p} + s}}}
    \le 
    C \seminorm[W^{s_1, p_1} (\Rset^N)]{u}^\theta \norm[L^1 (\Rset^N)]{\nabla u}^{1 - \theta},
    \qquad \forall u \in C^\infty_c (\Rset^N) \eofs ;
  \end{equation}
a slightly more careful argument shows that the constant \(C\) in \eqref{eq_phai1AichooTheeRah3usi2a} can be taken to depend only on \(N\) but not on \(p_1\) nor \(\theta\). Estimate \eqref{eq_phai1AichooTheeRah3usi2a} was announced in \cite{Brezis_VanSchaftingen_Yung_2021}*{Corollary 5.2}. The proofs of \cref{theorem_Lorentz} and \eqref{eq_phai1AichooTheeRah3usi2a} are given in \cref{section_lorentz}. In the same section we establish the optimality of the exponent \(\frac{p_1}{\theta}\) in \eqref{eq_giu7ohWoo2vae4jei0AhgieG}.

\medskip

To conclude this paper we mention another estimate in the spirit of Gagliardo--Nirenberg interpolation between \(L^\infty\) and \(W^{1, 1}\).
It is originally due to Figalli-Serra \cite[Lemma 3.1]{Figalli_Serra} when $p = 2$ and $q = \infty$, with roots in Figalli--Jerison \cite[Lemma 2.1]{FJ2014} (see also \cite[Lemma 2.2 and Corollary 2.3]{Gui_Li} for a simpler proof and more general version). 

\begin{theorem} \label{proposition_FS}
Let $N \geq 1$, $1 < p < \infty$ and $N < q \leq \infty$.
There exists a constant $C = C(N,p,q)$ such that 
\begin{multline} \label{eq:FS}
\iint\limits_{B_1 \times B_1} \frac{|u(x)-u(y)|^p}{|x-y|^{N+1}} \dif x \dif y\\[-1em]
\leq C  \|u\|_{L^{\infty}(B_1)}^{p - 1} \|\nabla u\|_{L^1(B_1)} \biggr(1 + \log \max\biggl\{\frac{\|\nabla u\|_{L^q(B_1)}}{\|u\|_{L^{\infty}(B_1)}},1\biggr\}\biggl)
\end{multline}
for every $u \in C^1(\overline{B_1})$.
\end{theorem}
Here $B_1$ denotes the unit ball in $\Rset^N$. See \cref{section_FS} for a proof of \cref{proposition_FS}.  
  
\subsection*{Acknowledgments}

This work was carried out during two visits of J. Van Schaftingen to Rutgers University. He thanks H. Brezis for the invitation and  the Department of Mathematics for its hospitality. P-L. Yung was partially supported by a Future Fellowship FT200100399 from the Australian Research Council.
 H. Brezis is grateful to C. Sbordone who communicated to him the interesting paper \cite{Greco_Schiattarella} by Greco and Schiattarella which triggered our work.

\section{Proofs of \texorpdfstring{\cref{thm_X_1,proposition_level_set_constant}}{Theorems 1 and 2}}
\label{section_constant}

  \cref{thm_X_1} is an immediate consequence of the standard \cref{lemma_level_set} below, and \cref{proposition_level_set_constant}.  
  \begin{lemma}
  \label{lemma_level_set}
   Let  \(f\) be a measurable function on \(\Rset^N \times \Rset^N\) in \(L^{p, q} (\Rset^N \times \Rset^N)\) with \(1 \le p < \infty\), \(1 \le q < \infty\). Then
   \begin{equation}
   \label{eq_Shu3Shas2to5haeVieC2eph8}
   \lim_{\lambda \to \infty} 
 \lambda^p 
 \mathcal{L}^{2N}
 \bigl(
 \bigl\{
 (x, y) \in \Rset^N \times \Rset^N
 \st \abs{f (x, y)}
 \ge \lambda
 \bigr\}
 \bigr)
  = 0\eofs .
  \end{equation}
  \end{lemma}
  \begin{proof}
   Set 
   \begin{equation*}
    \varphi (\lambda) = \int_{\lambda/2}^\lambda
    t^q
 \mathcal{L}^{2N}
 \bigl(
 \bigl\{
 (x, y) \in \Rset^N \times \Rset^N
 \st \abs{f (x, y)}
 \ge t
 \bigr\}
 \bigr)^{\frac{q}{p}} \frac{\dif t}{t}.
   \end{equation*}
Since \(f \in L^{p, q} (\Rset^N \times \Rset^N)\), we know that \(\varphi (\lambda) \to 0\) as \(\lambda \to \infty\).
On the other hand, 
\[
\begin{split}
 \varphi (\lambda) &
 \ge \mathcal{L}^{2N}
 \bigl(
 \bigl\{
 (x, y) \in \Rset^N \times \Rset^N
 \st \abs{f (x, y)}
 \ge \lambda
 \bigr\}
 \bigr)^{\frac{q}{p}} \int_{\lambda/2}^\lambda t^q \frac{\dif t}{t}\\
 & = \mathcal{L}^{2N}
 \bigl(
 \bigl\{
 (x, y) \in \Rset^N \times \Rset^N
 \st \abs{f (x, y)}
 \ge \lambda
 \bigr\}
 \bigr)^{\frac{q}{p}}
 \frac{\lambda^q}{q}
 \Bigl(1 - \frac{1}{2^q}\Bigr),
\end{split}
\]
which yields \eqref{eq_Shu3Shas2to5haeVieC2eph8}.
  \end{proof}
  
\begin{proof}[Proof of \cref{proposition_level_set_constant} when \(p > 1\)]
Let \(E_\lambda \subset \Rset^N \times \Rset^N\) denote the set in the left-hand side of \eqref{eq_EuK1che4iepahtheyoo2oiw3}. 
First observe that for each \(\lambda > 0\),
\[
 \iint\limits_{\Rset^N \times \Rset^N}
  \Biggl(
   \frac{\abs{u (x) - u (y)}}{\abs{x - y}^{\frac{N}{p} + 1}} - \lambda
  \Biggr)_+
  \dif y \dif x
  \le \int_\lambda^\infty \mathcal{L}^{2N} (E_t) \dif t
  \le \frac{1}{(p - 1)\lambda^{p - 1}} 
  \sup_{t \ge \lambda} t^p \mathcal{L}^{2N} (E_t) \eofs .
\]
Hence, we have 
\begin{equation}
\label{eq_hi3ohtohsh1iiXiz1miec1ee}
 \lim_{\lambda \to \infty}
 \lambda^{p - 1}
  \iint\limits_{\Rset^N \times \Rset^N}
  \Biggl(
   \frac{\abs{u (x) - u (y)}}{\abs{x - y}^{\frac{N}{p} + 1}} - \lambda
  \Biggr)_+
  \dif y \dif x= 0\eofs .
\end{equation}
We next use an argument similar to the one in \cites{488780,RanjbarMotlagh} and \cite{VanSchaftingen_2019}*{Proof of Proposition 5.1}.
From the triangle inequality and change of variable, we obtain
\begin{equation}
\label{eq_Sietheewu9uajuoB9ohsh5ah}
\begin{split} 
 \iint\limits_{\Rset^N \times \Rset^N}\!\!&
  \Biggl(
   \frac{\abs{u (x) - u (y)}}{\abs{x - y}^{\frac{N}{p} + 1}} - \lambda
  \Biggr)_+\!\!
  \dif y \dif x\\
& \le \!\!\iint\limits_{\Rset^N \times \Rset^N}
 \!\!\! \Biggl(
   \frac{\abs{u (x) - u (\frac{x + y}{2})}}{\abs{x - y}^{\frac{N}{p} + 1}} - \frac{\lambda}{2}
  \Biggr)_+\!\!\!\!
  \dif y \dif x
  + \!\!\! \iint\limits_{\Rset^N \times \Rset^N}\!\!\!
  \Biggl(
   \frac{\abs{u (\frac{x + y}{2}) - u (y)}}{\abs{x - y}^{\frac{N}{p} + 1}} - \frac{\lambda}{2}
  \Biggr)_+\!\!\!\!
  \dif y \dif x\\
& = 2^{\frac{N}{p} (p - 1)}\!\!\!\!
\iint\limits_{\Rset^N \times \Rset^N}\!\!\!\!
  \Biggl(
   \frac{\abs{u (x) - u (y)}}{\abs{x - y}^{\frac{N}{p} + 1}} - 2^\frac{N}{p} \lambda
  \Biggr)_+\!\!\!\!
  \dif y \dif x\eofs.
\end{split}
\end{equation}
Iterating \eqref{eq_Sietheewu9uajuoB9ohsh5ah}, we have in view of \eqref{eq_hi3ohtohsh1iiXiz1miec1ee},
\[
  \iint\limits_{\Rset^N \times \Rset^N}
  \Biggl(
   \frac{\abs{u (x) - u (y)}}{\abs{x - y}^{\frac{N}{p} + 1}} - \lambda
  \Biggr)_+
  \dif y \dif x = 0\eofs , \qquad \forall \lambda > 0\eofs ,
\]
from which it follows that \(u\) is constant.
\end{proof}

  \begin{proof}[Proof of \cref{proposition_level_set_constant} when \(p = 1\)]
As already mentioned, in this case the conclusion of \cref{proposition_level_set_constant} is essentially due to A. Poliakovsky \cite{Poliakovsky}. Indeed,  if a measurable function $u$ satisfies \eqref{eq_EuK1che4iepahtheyoo2oiw3}, then so does its truncation $u_h$ for any $h > 0$ where $u_h := \max\{\min\{u,h\},-h\}$. Then $u_h \in L^1(B_h)$, where $B_h := \{x \in \Rset^N \colon |x| < h\}$, and \cite{Poliakovsky}*{Cor.\ 1.1} (with $q = 1$, $\Omega = B_h$) shows that $u_h$ is a constant on $B_h$. Since this is true for every $h > 0$, this also shows $u$ is a constant.

The proof of \cite{Poliakovsky}*{Cor.\ 1.1} is quite intricate and we refer the reader to \cite{Poliakovsky}; it would be interesting to find a simpler argument as in the case \(p > 1\).
  \end{proof}
  
  We also call the attention of the reader to
  \begin{openproblem}
   Does the conclusion of \cref{proposition_level_set_constant} still hold if ``\(\lim\)'' is replaced by ``\(\liminf\)'' in \eqref{eq_EuK1che4iepahtheyoo2oiw3}?
  \end{openproblem}

\section{Proofs of \texorpdfstring{\cref{theorem_intro}}{Theorem 4}}
\label{section_weak_estimates}

\cref{theorem_intro} can be derived as an immediate consequence of \eqref{eq_chaK6eique5lo9ahy9Ooju9e} (applied with $p = 1$) and the fact that 
\[
\frac{|u(x)-u(y)|}{|x-y|^{\frac{N+1}{p}}} \geq \lambda \quad \text{implies} \quad \frac{|u(x)-u(y)|}{|x-y|^{N+1}} \geq \frac{\lambda^p}{(2 \|u\|_{L^{\infty}})^{p-1}}.
\]
Hence
\[
\mathcal{L}^{2N} \Big( \Big\{ (x,y) \in \Rset^N \times \Rset^N \colon \frac{|u(x)-u(y)|}{|x-y|^{\frac{N+1}{p}}} \geq \lambda \Big\} \Big) \leq \frac{2^{p-1} C}{\lambda^p} \|u\|_{L^{\infty}(\Rset^N)}^{p-1} \|\nabla u\|_{L^1(\Rset^N)},
\]
where $C = C(N)$ is as in \eqref{eq_chaK6eique5lo9ahy9Ooju9e}; note that $(2^{p-1} C)^{1/p}$ can be dominated by a constant depending only on $N$. This proves \cref{theorem_intro}. 
\qed

\bigskip

For the enjoyment of the reader we also present an elementary qualitative argument for the case \(p > 1\) of \cref{theorem_intro} which does not make use of \eqref{eq_chaK6eique5lo9ahy9Ooju9e}. It relies on the following estimate occuring in \cite{Bourgain_Brezis_Mironescu_2001}; unfortunately it yields a constant $C$ in \eqref{eq:1.14} which depends on $p$ and $N$, and which deteriorates as $p \searrow 1$.
Note that \eqref{eq_aesh0iichoo8OT0aa4eHiepu}
is a straightforward consequence of the inequality (see \cite{Brezis_2011}*{Proposition 9.3})
\[
\int_{\Rset^N}
\abs{u (x + h) - u (x)} \dif x
\le \abs{h} \int_{\Rset^N} \abs{\nabla u}\eofs,
\qquad \forall h \in \Rset^N\eofs,\ \forall u \in C^\infty_c (\Rset^N)\,.
\]

\begin{lemma}
\label{lemma_frac_1_rho}
For every \(u \in C^\infty_c (\Rset^N)\) and \(\rho \in L^1 (\Rset^N)\),
\begin{equation}
  \label{eq_aesh0iichoo8OT0aa4eHiepu}
\iint\limits_{\Rset^N \times \Rset^N}
\frac{\abs{u (x) - u (y)}}{\abs{x - y}} \rho (x - y) \dif y \dif x
\le 
\norm[L^1 (\Rset^N)]{\rho} \int_{\Rset^N} \abs{\nabla u}\eofs,
\end{equation}
and in particular choosing \(\rho (z) = \mathbf{1}_{B_r (0)} (z)/\abs{z}^{N - \delta}\), \(\delta > 0\), we obtain
\begin{equation}
  \label{eq_IeQuu6ahh5ohThieLaiphaiZ}
  \iint
    \limits_{\substack{(x, y) \in \Rset^N \times \Rset^N\\
    \abs{x - y} \le r}}
    \frac
      {\abs{u (x) - u (y)}}
      {\abs{x - y}^{N + 1 - \delta}}
      \dif y
      \dif x 
    \le
      C(N)
      \frac{r^{\delta}}{\delta}
      \int_{\Rset^N} \abs{\nabla u}
      \eofs.
\end{equation}
\end{lemma}

\begin{proof}%
  [Alternative proof of \eqref{eq:1.14} when \(p > 1\)]%
  \resetconstant%
 Define the set 
  \begin{equation}
    \label{eq_thaoWahnaesoa5Gi3taSah5E}
    E_\lambda
    \defeq \biggl\{
    (x, y) \in \Rset^N \times \Rset^N
    \st
    \frac{\abs{u (x) - u (y)}}{\abs{x - y}^{\frac{N + 1}{p}}} \ge \lambda
    \biggr\}\eofs.
  \end{equation}
  Observe that 
  \begin{equation*}
    E_\lambda \subseteq K_\lambda \defeq  \biggl\{(x, y) \in \Rset^N \times \Rset^N \st \abs{x - y} \le \bigl(2 \norm[L^\infty (\Rset^N)] {u}/\lambda\bigr)^{\frac{p}{N + 1}} \biggr\}\eofs.
  \end{equation*}
  Thus 
  \[
  \mathbf{1}_{E_\lambda} \le \mathbf{1}_{K_\lambda} \frac{1}{\lambda}  \frac{\abs{u (x) - u (y)}}{\abs{x - y}^{\frac{N + 1}{p}}}.
  \]
  Hence we have
  \[
  \begin{split}
    \mathcal{L}^{2 N} (E_{\lambda})
    \le
    \frac{1}{\lambda}
    \iint
    \limits_{K_\lambda}
    \frac
    {\abs{u (x) - u (y)}}
    {\abs{x - y}^{\frac{N + 1}{p}}}
    \dif y
    \dif x \eofs.
  \end{split}
  \]
  It then follows by \eqref{eq_IeQuu6ahh5ohThieLaiphaiZ}, with \(\delta \defeq (N + 1)(1 - \frac{1}{p}) > 0\) and \(r \defeq (2 \norm[L^\infty (\Rset^N)] {u}/\lambda)^{\frac{p}{N + 1}}\), that
  \begin{equation}
    \mathcal{L}^{2 N}
    \bigl(E_\lambda\bigr)
    \le
    \frac
    {
      C(N)
      \,
      (2\norm[L^\infty(\Rset^N)]{u})^{p - 1}
    }
    {(N + 1)(1 - \frac{1}{p}) \lambda^{p}}
    \int_{\Rset^N}
    \abs{\nabla u}
    \eofs.\qedhere
  \end{equation}
\end{proof}

\section{Optimality of \texorpdfstring{\cref{theorem_W11_Wsp,theorem_intro}}{Theorems 3 and 4} in the Lorentz scale}
\label{section_optimal}

\Cref{theorem_W11_Wsp,theorem_intro} cannot be improved. This is a consequence of the following lemma and its proof. 

\begin{lemma}
  \label{proposition_optimaity_linfty}
  Assume that \(1 \le p < \infty\).
  If
  \begin{equation}
  \label{eq_cheaghoochie6Eizier8Xoth}
  \Biggquasinorm[L^{p, q} (\Rset^N \times \Rset^N)]{\frac{u (x) - u(y)}{\abs{x - y}^{\frac{N + 1}{p}}} }
  \le 
  C
  \norm[L^\infty (\Rset^N)]{u}^{1 - \frac{1}{p}}
  \norm[L^1 (\Rset^N)]{\nabla u}^{\frac{1}{p}}\eofs,
  \qquad 
  \forall u \in C^\infty_c (\Rset^N)
  \end{equation}
holds for some \(1 \le q \le \infty\), then \(q = \infty\).
\end{lemma}

\begin{proof}
When \(p = 1\), the conclusion follows from \cref{thm_X_1}.
Indeed, we already know that if $q < \infty$, then for any measurable function \(u\),
\[
  \Biggquasinorm[L^{1, q} (\Rset^N \times \Rset^N)]{\frac{u (x) - u(y)}{\abs{x - y}^{N + 1}} } = \infty,
\]
unless \(u\) is a constant. 

When \(p > 1\), the argument is different since one may easily check that 
\[
 \Biggquasinorm[L^{p, q} (\Rset^N \times \Rset^N)]{\frac{u (x) - u (y)}{\abs{x - y}^{\frac{N + 1}{p}}} } < \infty, \qquad \forall u \in C^\infty_c (\Rset^N), \forall q \in [1,\infty].
\]

We consider the case \(N = 1\), the case \(N > 1\) being similar. 
By an approximation argument, it follows that \eqref{eq_cheaghoochie6Eizier8Xoth} holds for every \(u\in BV (\Rset)\) with compact support.
However, if \(u \defeq \chi_{[0, 1]}\), we have 
\begin{multline*}
\bigl\{ (x, y) \in (-1, 0) \times (0, 1) \st \abs{x - y} \le \lambda^{-p/2}\bigr\}\\
\subseteq 
E_\lambda
\defeq
\biggl\{
    (x, y) \in \Rset\times \Rset
    \st
    \frac{\abs{u (x) - u (y)}}{\abs{x - y}^{\frac{2}{p}}} \ge \lambda
    \biggr\}
\eofs,
\end{multline*}
and thus, if \(\lambda \ge 1\),
\[
\mathcal{L}^{2} (E_\lambda) \ge \frac{c}{\lambda^p}\eofs.
\]
Hence, if \(1 \le q < \infty\),
\[
  \Biggquasinorm[L^{p, q} (\Rset \times \Rset)]{\frac{u(x) - u(y)}{\abs{x - y}^{2/p}}}^q
= p \int_0^{\infty} \lambda^q \mathcal{L}^{2} (E_\lambda)^\frac{q}{p} \frac{\dif \lambda}{\lambda} \ge p \, c^\frac{q}{p} \int_1^\infty \frac{\dif \lambda}{\lambda}= \infty\eofs ,
\]
which contradicts \eqref{eq_cheaghoochie6Eizier8Xoth}.
\end{proof}

\section{Proof of \texorpdfstring{\cref{theorem_Lorentz}}{Theorem 5}}
\label{section_lorentz}

We deduce \cref{theorem_Lorentz} from Theorem 1.1 in \cite{Brezis_VanSchaftingen_Yung_2021} and the classical product property in Lorentz spaces.

\begin{proof}[Proof of \cref{theorem_Lorentz}]
By Theorem 1.1 in \cite{Brezis_VanSchaftingen_Yung_2021}, we have 
\[
\Biggquasinorm[L^{1, \infty} (\Rset^N \times \Rset^N)]{\frac{u (x) - u (y)}{\abs{x - y}^{N + 1}}}
\le 
C \norm[L^1 (\Rset^N)]{\nabla u}\eofs 
\]
where \(C = C(N)\). On the other hand, by definition of the Gagliardo semi-norm
\[
\Biggquasinorm[L^{p_1, p_1} (\Rset^N \times \Rset^N)]{\frac{u (x) - u (y)}{\abs{x - y}^{\frac{N}{p_1} + 1}}}
= 
\seminorm[W^{s_1, p_1}]{u}\eofs.
\]

We now observe that
\[
\frac{\abs{u (x) - u (y)}}{\abs{x - y}^{\frac{N}{p} + s}}
= \left(\frac{\abs{u (x) - u (y)}}{\abs{x - y}^{N + 1}}\right)^{1 - \theta}
\left(\frac{\abs{u (x) - u (y)}}{\abs{x - y}^{\frac{N}{p} + s_1 }}\right)^{\theta}.
\]
Hence, by the product property in Lorentz spaces \cite{ONeil_1963}*{Theorem 3.4},
\begin{multline} \label{eq:Lorentz_Holder_q}
  \Biggquasinorm[L^{p, \frac{p_1}{\theta}} (\Rset^N \times \Rset^N)]{\frac{u (x) - u (y)}{\abs{x - y}^{\frac{N}{p} + s }}}\\
  \le C\Biggquasinorm[L^{\frac{1}{1-\theta}, \infty} (\Rset^N \times \Rset^N)]{\biggl(\frac{\abs{u (x) - u (y)}}{\abs{x - y}^{N + 1}} \biggr)^{1 - \theta}}  \Biggquasinorm[L^{\frac{p_1}{\theta}, \frac{p_1}{\theta}} (\Rset^N \times \Rset^N)]{\biggl(\frac{\abs{u (x) - u (y)}}{\abs{x - y}^{\frac{N}{p} + 1}} \biggr)^{\theta}}\eofs.
\end{multline}
Here \(C = C(N,p_1,\theta)\). In order to conclude, we use the fact that for \(1 \le p <+\infty\), \(1 \le q \le \infty\) and \(0 < \beta < 1\), \(\quasinorm[L^{p/\beta, q/\beta}(\Rset^N \times \Rset^N)]{f^\beta} = \quasinorm[L^{p, q} (\Rset^N \times \Rset^N)]{f}^\beta\).
\end{proof}

If instead of \eqref{eq:Lorentz_Holder_q}, we use
\begin{multline} \label{eq:Lorentz_Holder}
\Biggquasinorm[M^p (\Rset^N \times \Rset^N)]{\frac{u(x)-u(y)}{|x-y|^{\frac{N}{p}+s}}}\\
\leq 2^{1/p} \Biggquasinorm[L^{\frac{1}{1-\theta},\infty} (\Rset^N \times \Rset^N)]{\biggl(\frac{\abs{u (x) - u (y)}}{\abs{x - y}^{N + 1}} \biggr)^{1 - \theta}} \Biggquasinorm[L^{\frac{p_1}{\theta},\infty} (\Rset^N \times \Rset^N)]{\biggl(\frac{\abs{u (x) - u (y)}}{\abs{x - y}^{\frac{N}{p} + 1}} \biggr)^{\theta}},
\end{multline}
then we obtain \eqref{eq_phai1AichooTheeRah3usi2a} with \(C = C(N)\) independent of \(p_1\) and \(\theta\).

The optimality of \cref{theorem_Lorentz} follows from

\begin{lemma}
  \label{proposition_critical_interpolation}
  Fix \(s_1 \in (0, 1)\), \(p_1 \in (1, \infty)\) such that \(s_1 p_1 \ge 1\) and \(\theta \in (0, 1)\).
  Let \(0 < s <1\) and \(1 < p<\infty\) be defined by \eqref{eq_Ahqueep9eci9aikio5eequai}.
If
\[
\Biggquasinorm[L^{p, q} (\Rset^N \times \Rset^N)]{\frac{u (x) - u(y)}{\abs{x - y}^{\frac{N}{p} + s}}}
\le C
\seminorm[W^{s_1, p_1}]{u}^{\theta}
\norm[L^1 (\Rset^N)]{\nabla u}^{1 - \theta},
\qquad
\forall u \in C^\infty_c (\Rset^N)
\]
holds for some \(1 \le q \le \infty\), then \(q \ge \frac{p_1}{\theta}\).
\end{lemma}

\begin{proof}[Proof of \cref{proposition_critical_interpolation} when \(s_1 p_1 = 1\)]
  We concentrate on the case \(N = 1\), the case \(N > 1\) being similar.
  Following \cite{Brezis_Mironescu_2018}*{Proof of Lemma 4.1, Step 1}, we define the function 
\[
u_k (x)= \varphi \bigl(k (\abs{x} - 1/2)\bigr)\eofs.
\]
where \(\varphi \in C^1 (\Rset)\), \(\varphi = 1\) on \((-\infty, -1]\) and \(\varphi = 0\) on \([1, \infty]\).
We have as in \cite{Brezis_Mironescu_2018} 
\begin{align}
  \label{eq_muSheigeerae0ahWei5pulah}
  \norm[L^1 (\Rset)]{u_k'} &\le C&
  &\text{and}&
  \seminorm[W^{s_1, p_1} (\Rset)]{u_k} &\le C (\log k)^\frac{1}{p_1}\eofs.
\end{align}
Given \(\lambda > 0\), we have since \(sp = 1\),
\begin{multline*}
  \Biggl\{ (x, y) \in [-1, 1] \times [-1, 1] \st \frac{\abs{u_k (x) - u_k (y)}}{\abs{x - y}^{\frac{1}{p} + s}} \ge \lambda\Biggr\}\\
\supseteq
\bigl\{ (x, y) \in [0, \tfrac{1}{2} - \tfrac{1}{k}] \times [\tfrac{1}{2} + \tfrac{1}{k}, 1] \st \abs{x - y} \le \lambda^{-p/2} \bigr\}\eofs.
\end{multline*}
Hence, there is \(c > 0\) such that if \(\lambda \le (k/4)^{2/p}\),
\[
\mathcal{L}^2 (E_\lambda) \ge \frac{c}{\lambda^p}\eofs. 
\]
It follows from \eqref{eq_chahvoChohyae4aiK9ez4zoh} that 
\begin{equation*}
  \Biggquasinorm[{L^{p, q} (\Rset \times \Rset)}]{\frac{u_k (x) - u_k (y)}{\abs{x - y}^{\frac{N}{p} + s}}}
  \ge \biggl(\int_1^{(k/4)^{2/p}} \frac{c \dif \lambda}{\lambda}\biggr)^\frac{1}{q}
\ge c' (\log k)^\frac{1}{q}\eofs.
\end{equation*}
By assumption and by \eqref{eq_muSheigeerae0ahWei5pulah}, we have 
\[
(\log k)^\frac{1}{q}
\le C (\log k)^\frac{\theta}{p_1}\eofs,
\]
and it follows thus that \(q \ge \frac{p_1}{\theta}\).
\end{proof}

\begin{proof}[Proof of \cref{proposition_critical_interpolation} for \(s_1 p_1 > 1\)]
We concentrate on the case \(N = 1\), the case \(N > 1\) being similar.
We adapt the proof from \cite{Brezis_Mironescu_2018}*{Proof of Lemma 4.1}, where functions \(w_j^k\) are constructed (in Step 2 there, with \(\alpha \defeq (s_1 - \frac{1}{p_1})/(1 - \frac{1}{p_1}) = (s - \frac{1}{p})/(1 - \frac{1}{p}) \))
and satisfy  
\begin{align}
  \label{eq_Gothie7dush2aimisahjoo7a}
   \norm[{L^1([0, 1])}]{(w_j^k)'} &= 1,&
  \limsup_{k \to \infty} \seminorm[{W^{s_1, p_1} ([0, 1])}]{w_j^k} &\le C j^{1/p_1}
\end{align}
and 
\begin{equation}
  \label{eq_jach8aekeeXiicee9aem1eej}
  \limsup_{k \to \infty} \seminorm[{W^{s, p} ([0, 1])}]{w_j^k} \ge \frac{j^{1/p}}{C}\eofs.
\end{equation}
We improve \eqref{eq_jach8aekeeXiicee9aem1eej} to cover the case \(q \ne p\) in the Lorentz scale \(L^{p, q}\).

Given \(\lambda > 0\), we have 
\begin{multline*}
  \biggl\{ (x, y) \in [0, 1] \times [0, 1] \st \frac{\abs{w_j^k (x) - w_j^k (y)}}{\abs{x - y}^{\frac{1}{p} + s}} \ge \lambda \biggr\}\\
\supseteq \bigl\{ (x, y) \in [0, 1] \times [0, 1] \st \abs{w_j^k (x) - w_j^k (y)} \ge \lambda\bigr\}\eofs,
\end{multline*}
and thus if \(\lambda \le \frac{1}{3}\), we have 
\[
\mathcal{L}^2 \biggl( \biggl\{ (x, y) \in [0, 1] \times [0, 1] \st \frac{\abs{w_j^k (x) - w_j^k (y)}}{\abs{x - y}^{\frac{1}{p} + s}} \ge \lambda \biggr\} \biggr)
\ge c
\]
for some constant \(c > 0\).

Next by the inductive definition of \(w_j^k\) and by scaling, we have 
\begin{equation*}
    \begin{split}
  \mathcal{L}^2 \biggl( \biggl\{ (x, y) &\in [0, 1] \times [0, 1] \st \frac{\abs{w_j^k (x) - w_j^k (y)}}{\abs{x - y}^{\frac{1}{p} + s}} \ge \lambda \biggr\} \biggr)\\
  &\ge \sum_{\ell = 1}^k  
  \mathcal{L}^2 \biggl( \biggl\{ (x, y) \in I_k^\ell \times I_k^\ell \st \frac{\abs{w_j^k (x) - w_j^k (y)}}{\abs{x - y}^{\frac{1}{p} + s}} \ge \lambda \biggr\} \biggr)\\
  &\ge \frac{1}{k^{\frac{2}{\alpha} - 1}}
  \mathcal{L}^2 \biggl( \biggl\{ (x, y) \in [0, 1] \times [0, 1] \st \frac{\abs{w_{j - 1}^k (x) - w_{j - 1}^k (y)}}{\abs{x - y}^{\frac{1}{p} + s}} \ge  \frac{\lambda}{k^{\frac{1}{p}(\frac{2}{\alpha} - 1)}} \biggr\} \biggr)\eofs,
  \end{split}
\end{equation*}
where we used
\(
\alpha = \frac{s - \frac{1}{p}}{1 - \frac{1}{p}}
\)
to obtain
\[
 \frac{\frac{1}{p} + s}{\alpha}-
1
= \frac{2}{p\alpha} + \frac{s - \frac{1}{p}}{\alpha} - 1
= \frac{1}{p} \biggl(\frac{2}{\alpha} - 1\biggr)\eofs.
\]

By induction, for each \(i \in \{1, \dotsc, j\}\) and \(
\lambda \le k^{\frac{i - 1}{p}(\frac{2}{\alpha} - 1)}/3\),
we have 
\[  
\mathcal{L}^2 \biggl( \biggl\{ (x, y) \in [0, 1] \times [0, 1] \st \frac{\abs{w_j^k (x) - w_j^k (y)}}{\abs{x - y}^{\frac{1}{p} + s}} \ge \lambda \biggr\} \biggr)
\ge 
\frac{c}{k^{(i - 1)(\frac{2}{\alpha} - 1)}}\eofs.
\]
We finally estimate in view of \eqref{eq_chahvoChohyae4aiK9ez4zoh}
\begin{equation}
  \label{eq_Eegh1eiquucouphaepon4ahl}
  \Biggquasinorm[{L^{p, q} ([-1, 1] \times [-1, 1])}]{\frac{w^k_j (x) - w^k_j (y)}{\abs{x - y}^{\frac{1}{p} + s}}}
   \ge
c'
\left( \sum_{i = 1}^j \int_{k^{\frac{i-1}{p}(\frac{2}{\alpha} - 1)}/3}^{k^{\frac{i}{p}(\frac{2}{\alpha} - 1)}/3}
\frac{\lambda^{q- 1}}{k^{\frac{q}{p}(i - 1)(\frac{2}{\alpha} - 1)}} \dif \lambda \right)^{1/q}
\ge c'' j^{1/q} \eofs.
\end{equation}
The conclusion follows from the assumptions combined with the estimates \eqref{eq_Gothie7dush2aimisahjoo7a} and \eqref{eq_Eegh1eiquucouphaepon4ahl}.
\end{proof}

\section{Proof of \texorpdfstring{\cref{proposition_FS}}{Theorem 6}} \label{section_FS}

We may always extend \(u\) to \(B_3\) with control of norms and assume that \(\Vert u\Vert_{L^\infty (B_1)} = 1\).
By the Sobolev--Morrey embedding we have (since \(q > N\)) 
\begin{equation}
\label{eq_hohx0Ate4Aezoh7ooy2os4jo}
\abs{u (x) - u (y)} \le C \min \bigl\{1, \abs{x - y}^\alpha \Vert \nabla u \Vert_{L^q (B_1)}\bigr\},\quad 
\text{for all \(x, y \in B_1\),}
\end{equation}
where \(\alpha = 1 - \tfrac{N}{q}\).
Thus 
\[
 \abs{u (x) - u (y)}^p \le C  \min \bigl\{1, \abs{x - y}^{\alpha(p - 1)} \Vert \nabla u \Vert_{L^q (B_1)}^{p - 1}\bigr\} \abs{u (x) - u (y)}
\]
and therefore
\begin{multline}
 \iint\limits_{B_1 \times B_1} \frac{|u(x)-u(y)|^p}{|x-y|^{N+1}} \dif x \dif y\\
 \le C \int_{B_2} \dif h \int_{B_1} \min \bigl\{1, \abs{h}^{\alpha(p - 1)} \Vert \nabla u \Vert_{L^q (B_1)}^{p - 1}\bigr\} \frac{\abs{u (x +h ) - u(x)}}{|h|^{N+1}}\dif x.
\end{multline}
Since
\[
 \int_{B_1} \abs{u (x + h) - u (x)} \dif x \le \abs{h} \lVert \nabla u \rVert_{L^1 (B_3)} \le C \abs{h} \lVert \nabla u \rVert_{L^1 (B_1)},
 \quad \text{for all \(h \in B_2\)},
\]
we conclude that 
\begin{equation}
\begin{split}
 \smashoperator{\iint\limits_{B_1 \times B_1}}\frac{|u(x)-u(y)|^p}{|x-y|^{N+1}} \dif x \dif y
&\le C \lVert \nabla u \rVert_{L^1 (B_1)} \!\int_{B_2}\!\!\min \bigl\{1, \abs{h}^{\alpha(p - 1)} \Vert \nabla u \Vert_{L^q (B_1)}^{p - 1}\bigr\} \frac{\dif h}{\abs{h}^N}\\
&= C' \lVert \nabla u \rVert_{L^1 (B_1)} \!\int_0^2\!\! \min \bigl\{1, r^{\alpha(p - 1)} \Vert \nabla u \Vert_{L^q (B_1)}^{p - 1}\bigr\} \frac{\dif r}{r}
\end{split}
\end{equation}
and the conclusion follows from a straightforward computation.
\qed

\end{document}